\documentclass[11pt]{article}

\usepackage[a4paper,margin=28mm]{geometry}
\usepackage{amsmath,amssymb,amsthm,mathtools}
\usepackage{microtype}
\usepackage{xcolor}
\usepackage[colorlinks=true,linkcolor=blue!55!black,citecolor=blue!55!black,urlcolor=blue!55!black]{hyperref}
\hypersetup{
  pdftitle={A Uniform Proof for the Small Davenport Constant of the Exponent-p Heisenberg Group},
  pdfauthor={Andreas Volkmann},
  pdfsubject={Small Davenport constants and finite Heisenberg groups}
}

\newtheorem{theorem}{Theorem}[section]
\newtheorem{lemma}[theorem]{Lemma}

\newtheorem{corollary}[theorem]{Corollary}
\theoremstyle{definition}

\theoremstyle{remark}
\newtheorem{remark}[theorem]{Remark}

\newcommand{\Fp}{\mathbb F_p}

\newcommand{\supp}{\operatorname{supp}}

\newcommand{\sunion}{\mathbin{\mathaccent\cdot\cup}}

\title{A Uniform Proof for the Small Davenport Constant\\
of the Exponent-$p$ Heisenberg Group}
\author{Andreas Volkmann}
\date{19 August 2026}

\begin{document}
\maketitle

\begin{abstract}
Let $p$ be an odd prime and let $H_{p^3}=\operatorname{UT}_3(\Fp)$ be the
Heisenberg group of order $p^3$ and exponent $p$.  We prove
\[
  \mathsf d(H_{p^3})=3p-3.
\]
The main ingredient of the proof is an order-value growth theorem.  If $B$ is a
noncollinear zero-sum sequence of $n$ nonzero vectors in $\Fp^2$, then the
alternating areas obtained by ordering $B$ assume at least
$\min(p,n-1)$ distinct values.  Its proof is a short contraction induction:
contract a suitable independent pair, replace the contracted vector in both
orders, and apply Cauchy--Davenport.  A polynomial relative-subsum theorem and
a sharp representation-rigidity lemma then turn this local growth into a
uniform spread bound.  Combined with the standard product-one criterion for
$H_{p^3}$, the spread bound yields the upper bound; the usual sequence
$x^{p-1}y^{p-1}v^{p-1}$ gives the lower bound.
\end{abstract}

\section{Introduction}

For a finite group $G$, a sequence is a finite unordered list of group
elements, with repetitions allowed.  A nonempty subsequence is
\emph{product-one} if its terms can be ordered so that their product is the
identity.  The small Davenport constant $\mathsf d(G)$ is the maximum length
of a product-one-free sequence over $G$.

We write
\[
H_{p^3}=\operatorname{UT}_3(\Fp)
 =\left\{M(a,b,c)=
 \begin{pmatrix}1&a&c\\0&1&b\\0&0&1\end{pmatrix}:a,b,c\in\Fp\right\},
\]
with multiplication
\[
 M(a,b,c)M(a',b',c')=M(a+a',b+b',c+c'+ab').
\]
Godara and Sarkar proved the case $p=3$ and posed the equality
$\mathsf d(H_{p^3})=3p-3$ for every odd prime \cite{GodaraSarkar}.
In recent preprints, White proved the case $p=5$ \cite{White}, and Volkmann
proved the case $p=7$ \cite{Volkmann}.  Both later proofs use the product-one
criterion and spread framework recalled below, followed by $p$-specific finite
verification.

The purpose of this note is to give a uniform proof of the asserted equality.
Two elementary
ingredients do the work.  The first is a polynomial lower bound for subsums
that land on a prescribed line.  The second is a contraction theorem for the
number of ordering defects of a noncollinear zero-sum block.  Once these are
combined, no classification of long zero-sum-free sequences and no finite
search is needed.

\section{Order defects and alternating area}

Let $V=\Fp^2$, and write a vector as $u=(a(u),b(u))$.  Let
\[
  \omega(u,v)=a(u)b(v)-a(v)b(u)
\]
be the standard alternating form.  For an ordering
$I=(u_1,\ldots,u_n)$ define
\[
 q(I)=\sum_{i<j}a(u_i)b(u_j),\qquad
 W(I)=\sum_{i<j}\omega(u_i,u_j).
\]
For an unordered sequence $B$, let
\[
 \Omega(B)=\{q(I): I\text{ is an ordering of }B\},\qquad
 \mathcal W(B)=\{W(I): I\text{ is an ordering of }B\}.
\]

\begin{lemma}[Affine equivalence]\label{lem:affine}
For every sequence $B$ over $V$,
\[
  2q(I)=W(I)+K_B,
  \qquad
  K_B=\Big(\sum_{u\in B}a(u)\Big)
      \Big(\sum_{u\in B}b(u)\Big)-\sum_{u\in B}a(u)b(u),
\]
where $K_B$ is independent of the ordering.  Hence, for odd $p$,
\[
  |\Omega(B)|=|\mathcal W(B)|.
\]
\end{lemma}

\begin{proof}
Let $q^{\mathrm{op}}(I)=\sum_{i<j}a(u_j)b(u_i)$.  Then
$q+q^{\mathrm{op}}=K_B$ and $q-q^{\mathrm{op}}=W$.  Adding these identities
gives the formula.  Multiplication by $2$ is bijective in $\Fp$.
\end{proof}

We call a vector sequence \emph{collinear} if all its terms lie on a
one-dimensional subspace, and \emph{mixed} otherwise.

\section{The order-value growth theorem}

The following selection observation is the geometric core of the proof.

\begin{lemma}[Independent contraction]\label{lem:selection}
Let $B$ be a mixed zero-sum sequence over $V\setminus\{0\}$ with
$|B|\ge4$.  There are independent occurrences $s,y\in B$ such that
$B\setminus\{s,y\}$ is mixed.  In particular, $s+y\ne0$, and
\[
  B'=(B\setminus\{s,y\})\sunion\{s+y\}
\]
is again a mixed zero-sum sequence over $V\setminus\{0\}$.
\end{lemma}

\begin{proof}
Choose independent occurrences $a,b\in B$, and put
$R=B\setminus\{a,b\}$.  If $R$ is mixed, take $(s,y)=(a,b)$.

Suppose that $R$ lies on a line $L$.  Since $B$ has sum zero,
$a+b=-\sum R\in L$.  Independence of $a,b$ implies that neither $a$ nor $b$
lies on $L$.  Choose $c\in R$.  Then $a,c$ are independent.  If $|B|\ge5$,
the complement $B\setminus\{a,c\}$ contains $b\notin L$ and at least one
nonzero term of $R\setminus\{c\}\subset L$, so it is mixed.  If $|B|=4$,
write $R=\{c,d\}$; then the complement $\{b,d\}$ is independent for the
same reason.  Thus $(s,y)=(a,c)$ works in both cases.

The contracted sequence has the same total sum as $B$.  It contains the
mixed complement $B\setminus\{s,y\}$, and $s+y\ne0$ because $s,y$ are
independent.
\end{proof}

\begin{theorem}[Order-value growth]\label{thm:growth}
Let $p$ be odd, and let $B$ be a mixed zero-sum sequence of $n\ge3$ nonzero
vectors in $\Fp^2$.  Then
\[
  |\Omega(B)|=|\mathcal W(B)|\ge \min(p,n-1).
\]
\end{theorem}

\begin{proof}
We induct on $n$.  For $n=3$, write
$B=\{x,y,-x-y\}$ with $\omega(x,y)=d\ne0$.  Its ordering values are
$\mathcal W(B)=\{d,-d\}$, so the assertion holds.

Let $n\ge4$.  Use Lemma~\ref{lem:selection} to choose $s,y$, put
$t=s+y$, and let
\[
  B'=(B\setminus\{s,y\})\sunion\{t\}.
\]
Set $d=\omega(s,y)\ne0$.  In any ordering of $B'$, replace the occurrence
of $t$ by the adjacent block $(s,y)$ or $(y,s)$.  By bilinearity, every
pairing with terms outside the block is unchanged; only the internal term is
new.  Therefore
\[
  \mathcal W(B)\supseteq \mathcal W(B')+\{d,-d\}.
\]
Since $p$ is odd, the second summand has two elements.  Cauchy--Davenport
\cite{Nathanson} and
the induction hypothesis give
\begin{align*}
 |\mathcal W(B)|
 &\ge \min\bigl(p,|\mathcal W(B')|+1\bigr)\\
 &\ge \min(p,n-1).
\end{align*}
Lemma~\ref{lem:affine} transfers the same bound to $\Omega(B)$.
\end{proof}

\begin{remark}
The bound is sharp.  For independent $u,v$, put $d=\omega(u,v)$ and let
$1\le m\le p-2$.  Directly sorting the positions of $v$ and $-mu-v$ among
the $m$ copies of $u$ gives
\[
 \mathcal W\bigl(u^m v(-mu-v)\bigr)
 =d\{-m,-m+2,\ldots,m\}.
\]
These $m+1$ values are distinct.  Thus this sequence has length $m+2$ and
exactly $m+1$ ordering values; at $m=p-2$ it has $p-1$ values.
\end{remark}

\section{A relative subsum theorem}

For a sequence $A$ in an abelian group, let $\Sigma_0(A)$ be the set of all
subsequence sums, including the empty sum.  A sequence is zero-sum-free if
no nonempty subsequence has sum zero.

\begin{theorem}[Relative subsums]\label{thm:RH}
Let $A=(a_1,\ldots,a_N)$ be a zero-sum-free sequence over $V=\Fp^2$, and let
$H\le V$ be a line.  Then
\[
  |\Sigma_0(A)\cap H|\ge N-p+2.
\]
No condition on $A\cap H$ is required.
\end{theorem}

\begin{proof}
Choose linear forms $q,\ell:V\to\Fp$ such that
$\ker q=H$ and $\ell|_H:H\to\Fp$ is an isomorphism.  Then
$v\mapsto(\ell(v),q(v))$ is injective.  Give each occurrence $a_i$ its own
Boolean variable $X_i$, and put
\[
 Q(X)=\sum_iq(a_i)X_i,\qquad L(X)=\sum_i\ell(a_i)X_i.
\]
Define the set of distinct nonzero trace values
\[
 R=\{L(x):x\in\{0,1\}^N\setminus\{0\},\ Q(x)=0\}\subseteq\Fp^*.
\]
The inclusion in $\Fp^*$ follows from zero-sum-freeness and injectivity of
$(\ell,q)$.  Let
\[
 P(T)=\prod_{r\in R}(T-r),\qquad
 F(X)=\bigl(1-Q(X)^{p-1}\bigr)P(L(X)).
\]
On the Boolean cube, $F$ vanishes at every nonzero point: if $Q\ne0$, use
Fermat's theorem; if $Q=0$, then $L\in R$.  At the origin,
$F(0)=P(0)\ne0$.  Hence the unique multilinear representation of $F$ on the
Boolean cube is
\[
  P(0)\prod_{i=1}^N(1-X_i),
\]
which has degree exactly $N$.  Multilinear reduction modulo
$X_i^2-X_i$ never increases degree, whereas the displayed polynomial $F$
has degree at most $p-1+|R|$.  Consequently
\[
 N\le p-1+|R|.
\]
The empty sum supplies the additional value zero in $\Sigma_0(A)\cap H$, so
$|\Sigma_0(A)\cap H|=1+|R|\ge N-p+2$.
\end{proof}

\begin{corollary}\label{cor:D}
The Davenport constant of $\Fp^2$ is
\[
 D(\Fp^2)=2p-1.
\]
Moreover, every zero-sum-free sequence $A$ of length $2p-2$ is
subsum-complete: $\Sigma_0(A)=\Fp^2$.
\end{corollary}

\begin{proof}
The first assertion is Olson's classical theorem
\cite{OlsonI,OlsonII}; we include the short deduction because it follows
directly from Theorem~\ref{thm:RH}.
If $A$ were zero-sum-free of length $2p-1$, Theorem~\ref{thm:RH} would give
$|\Sigma_0(A)\cap H|\ge p+1$ inside the $p$-element line $H$.  The sequence
$e_1^{p-1}e_2^{p-1}$ gives the matching lower bound.

Now let $|A|=2p-2$ and $g\ne0$.  The sequence
$A\sunion\{-g\}$ has length $D(\Fp^2)$ and hence a
nonempty zero-sum subsequence.  It must contain $-g$, because $A$ is
zero-sum-free.  Its remaining terms sum to $g$.
\end{proof}

We also need the equality case of the elementary one-dimensional subsum
bound.

\begin{lemma}[Cyclic equality case]\label{lem:cyclic}
Let $X$ be a zero-sum-free sequence of $h\le p-1$ nonzero elements of
$\Fp$.  Then $|\Sigma_0(X)|\ge h+1$.  If equality holds, all terms of $X$
are equal.
\end{lemma}

\begin{proof}
The inequality follows by iterating Cauchy--Davenport on
$\{0\}+\{0,x_1\}+\cdots+\{0,x_h\}$.

Assume equality.  In any ordering of $X$, the $h+1$ prefix sums are distinct
and therefore exhaust $\Sigma_0(X)$.  Swap two adjacent terms.  All prefix
sums except the one between those two terms stay fixed.  The new prefix sums
are again distinct and exhaust the same set.  Thus the changed prefix sum
must equal the old one, and the two adjacent terms are equal.  Adjacent swaps
show that all terms are equal.
\end{proof}

\begin{lemma}[Representation rigidity]\label{lem:rigidity}
Let $A$ be a zero-sum-free sequence of length $2p-2$ over $\Fp^2$, and let
$u\ne0$.  Suppose every nonempty subsequence of $A$ with sum $u$ is a
singleton.  Then $u$ occurs exactly $p-1$ times in $A$.
\end{lemma}

\begin{proof}
By Corollary~\ref{cor:D}, $u$ is represented, so $u$ occurs in $A$.  Set
$H=\langle u\rangle$, let $S=A\cap H$, $h=|S|$, and $C=A\setminus S$.
Since $S$ is zero-sum-free in $H\cong C_p$, one has $h\le p-1$.
Let
\[
 R=\bigl(\Sigma_0(C)\cap H\bigr)\setminus\{0\}.
\]
Theorem~\ref{thm:RH} gives
\[
 |R|\ge |C|-p+1=p-1-h.
\]
Zero-sum-freeness makes $R$ and $-\Sigma_0(S)$ disjoint.  On the other hand,
Lemma~\ref{lem:cyclic} gives $|\Sigma_0(S)|\ge h+1$.  Their lower bounds add
to $p=|H|$, so equality holds throughout and
\[
 R=H\setminus\bigl(-\Sigma_0(S)\bigr),
 \qquad |\Sigma_0(S)|=h+1.
\]
By Lemma~\ref{lem:cyclic} and the occurrence of $u$, one has $S=u^h$.

If $h\le p-2$, the representation hypothesis implies $u\notin R$: no term
of $C$ equals $u$, so a $C$-representation could not be a singleton.  Hence
$u\in-\Sigma_0(S)=\{-ju:0\le j\le h\}$.  This forces
$j\equiv-1\pmod p$, impossible for $j\le p-2$.  Thus $h=p-1$.
\end{proof}

\section{A uniform quotient theorem}

\begin{theorem}[Uniform block theorem]\label{thm:block}
Let $p$ be odd and $1\le k\le p$.  Let $Q$ be a sequence over
$\Fp^2\setminus\{0\}$ such that
\[
 |Q|=2p+k-2,
 \qquad
 |Q\cap L|\le p+k-2
 \quad\text{for every line }L.
\]
Then $Q$ contains a mixed zero-sum block $U$ satisfying
\[
  |\Omega(U)|\ge k.
\]
\end{theorem}

\begin{proof}
We first prove that $Q$ contains a mixed zero-sum block.  Since
$|Q|\ge2p-1=D(\Fp^2)$, it has a nonempty zero-sum block.  If none is mixed,
choose such a block on a line $L$.  Put $h=|Q\cap L|$ and
$A=Q\setminus(Q\cap L)$.  The sequence $A$ is zero-sum-free: a zero-sum
block in $A$ would either be mixed, or lie on a second line and combine with
the chosen $L$-block to form a mixed zero-sum block.

By Theorem~\ref{thm:RH}, the number of nonzero $L$-valued subsums of $A$ is
at least
\[
 |A|-p+1=p+k-1-h.
\]
The nonzero sequence $Q\cap L$ has at least $\min(h,p-1)$ distinct nonzero
subsum values, by Cauchy--Davenport.  If $h\le p-1$, the two bounds add to
$p+k-1>p-1$; if $h\ge p-1$, the line subsums already contain every nonzero
value of $L$.  In either case, a nonzero trace of $A$ is cancelled by a
nonempty subsum of $Q\cap L$.  Their union is a mixed zero-sum block, a
contradiction.

Start with one mixed zero-sum block and repeatedly remove any further
zero-sum block from the remaining sequence until the residual sequence $A$
is zero-sum-free.  Let $U$ be the union of all removed blocks.  Then $U$ is
itself a mixed zero-sum block.  Since $|A|\le2p-2$,
\[
 |U|=|Q|-|A|\ge k.
\]
If $|U|\ge k+1$, Theorem~\ref{thm:growth} gives $|\Omega(U)|\ge k$, and we
are done.  Thus it remains to consider $|U|=k$; in particular,
$|A|=2p-2$.  The cases $k\le2$ cannot occur because a mixed zero-sum block
has length at least three.

Assume for a contradiction that $Q$ contains no mixed zero-sum block $E$
with $|\Omega(E)|\ge k$.  In particular $|\Omega(U)|\le k-1$, and the
growth theorem then gives equality $|\Omega(U)|=k-1$.
By Corollary~\ref{cor:D}, $A$ is
subsum-complete.  Fix an occurrence $x\in U$ and any subsequence $T\subseteq A$
with sum $x$.  The sequence
\[
 E=(U\setminus\{x\})\sunion T
\]
is zero-sum and mixed.  Indeed, $U\setminus\{x\}$ cannot be collinear: if it
lay on a line $L$, then its sum $-x$ would put $x$ on $L$ as well.
If $|T|\ge2$, then $|E|\ge k+1$, and Theorem~\ref{thm:growth} gives
$|\Omega(E)|\ge k$, contradicting the assumption.  Therefore every
representation of every support value $x$ of $U$ by a subsequence of $A$ is
a singleton.

Lemma~\ref{lem:rigidity} now gives
\[
  \mathsf v_x(A)=p-1
  \qquad\text{for every }x\in\supp(U).
\]
As $|A|=2p-2$, the support of $U$ has at most two elements.  Since $U$ is
mixed, it has exactly two independent support values, say
$U=u^a v^b$ with $a,b>0$.  The equation $au+bv=0$ forces
$p\mid a$ and $p\mid b$, hence $|U|\ge2p$.  This contradicts
$|U|=k\le p$.
\end{proof}

\section{The Heisenberg group}

Let $\pi(M(a,b,c))=(a,b)$.  The center is
$Z(H_{p^3})=\{M(0,0,c):c\in\Fp\}\cong C_p$.

\begin{lemma}[Product-one criterion]\label{lem:criterion}
For an ordered list $I=(g_1,\ldots,g_m)$ with
$g_i=M(a_i,b_i,c_i)$,
\[
 \prod_{i=1}^m g_i
 =M\left(\sum_i a_i,\sum_i b_i,
          \sum_i c_i+\sum_{i<j}a_i b_j\right).
\]
Thus a block has a product-one ordering precisely when its projection sums
to zero and the negative sum of its individual central coordinates belongs
to its cross-sum set $\Omega$.
\end{lemma}

\begin{proof}
Induction on $m$ using the multiplication law.
\end{proof}

\begin{theorem}\label{thm:main}
For every odd prime $p$,
\[
  \boxed{\mathsf d(H_{p^3})=3p-3.}
\]
\end{theorem}

\begin{proof}
For the lower bound, let
\[
 x=M(1,0,0),\qquad y=M(0,1,0),\qquad v=M(0,0,1).
\]
The sequence $x^{p-1}y^{p-1}v^{p-1}$ is product-one-free.  Indeed, a
product-one subsequence using $a$ copies of $x$ and $b$ copies of $y$ must
have $a\equiv b\equiv0\pmod p$.  Since $0\le a,b\le p-1$, this gives
$a=b=0$.  The remaining central subsequence $v^c$ is the identity only when
$c=0$.  Hence $\mathsf d(H_{p^3})\ge3p-3$.

For the upper bound, suppose that a product-one-free sequence $S$ of length
$3p-2$ exists.  Let $z$ be the number of central terms, and let $Q$ be the
sequence of nonzero projections of the remaining terms.  The central part is
zero-sum-free in $C_p$, so $z\le p-1$.  Set
\[
  k=p-z,
  \qquad
  |Q|=3p-2-z=2p+k-2.
\]

For every line $L\le\Fp^2$, the inverse image $\pi^{-1}(L)$ is abelian because
commutators are determinants of projected vectors; it has order $p^2$ and
exponent $p$, hence isomorphic to $C_p^2$.  It
contains all $z$ central terms and all terms whose projections lie on $L$.
Corollary~\ref{cor:D} and product-one-freeness therefore imply
\[
  z+|Q\cap L|\le2p-2,
  \qquad
  |Q\cap L|\le p+k-2.
\]
Theorem~\ref{thm:block} supplies a mixed zero-sum block $U\subseteq Q$ with
$|\Omega(U)|\ge k$.

Choose the corresponding noncentral group terms.  Their fixed individual
central coordinates translate $\Omega(U)$ but do not change its cardinality.
The central part of $S$ has a subset-sum set of size at least $z+1$, by the
one-dimensional Cauchy--Davenport bound.  A final application of
Cauchy--Davenport gives
\[
 \left|\Omega(U)+\Sigma_0(S\cap Z(H_{p^3}))\right|
 \ge \min(p,k+z)=p.
\]
After including the fixed central-coordinate translation, some ordering of
the chosen noncentral block and some subset of the central terms have total
central coordinate zero.  Their projection sum is already zero, so
Lemma~\ref{lem:criterion} gives a nonempty product-one subsequence.  This
contradicts the choice of $S$ and proves the upper bound.
\end{proof}

\section{Concluding remarks}

The proof is entirely theoretical.  Finite computations were used only as
adversarial consistency checks during discovery; no computation is part of
the logical argument.  The proof itself uses only Cauchy--Davenport,
elementary linear algebra, and the Boolean-cube degree argument given above.

\paragraph{AI-assisted tools disclosure.}
AI-assisted tools were used during exploration, drafting, and adversarial
consistency checking.  The mathematical argument presented here is
self-contained; no output of an AI system is invoked as an external authority.


\begin{thebibliography}{9}

\bibitem{GodaraSarkar}
N.~K. Godara and S. Sarkar,
\emph{A note on a conjecture of Gao and Zhuang for groups of order $27$},
J. Algebra Appl. 24 (2025), no. 11, Paper 2550268;
\href{https://doi.org/10.1142/S0219498825502688}{doi:10.1142/S0219498825502688};
\href{https://arxiv.org/abs/2311.02387}{arXiv:2311.02387}.

\bibitem{White}
P. White,
\emph{The small Davenport constant of the Heisenberg group of order $125$},
\href{https://arxiv.org/abs/2607.14379}{arXiv:2607.14379} (2026).

\bibitem{Volkmann}
A. Volkmann,
\emph{The small Davenport constant of the Heisenberg group of order $343$},
\href{https://arxiv.org/abs/2608.13747}{arXiv:2608.13747} (2026).

\bibitem{OlsonI}
J.~E. Olson,
\emph{A combinatorial problem on finite Abelian groups. I},
J. Number Theory 1 (1969), 8--10.

\bibitem{OlsonII}
J.~E. Olson,
\emph{A combinatorial problem on finite Abelian groups. II},
J. Number Theory 1 (1969), 195--199.

\bibitem{Nathanson}
M.~B. Nathanson,
\emph{Additive Number Theory: Inverse Problems and the Geometry of Sumsets},
Graduate Texts in Mathematics 165, Springer, 1996.

\end{thebibliography}
\end{document}